\documentclass[a4paper,abstracton]{scrartcl}
\usepackage[utf8]{inputenc}
\usepackage{amsfonts,amsthm,amssymb,amsmath}
\usepackage[inline]{enumitem}
\usepackage{color}
\usepackage{graphicx} 
\usepackage{authblk}

\newtheorem{thm}{Theorem}
\newtheorem{lem}[thm]{Lemma}

\theoremstyle{remark}

\DeclareMathOperator{\Aut}{Aut}

\title{Breaking small automorphisms by list colourings}
\author{Jakub Kwaśny, Marcin Stawiski\footnote{ Corresponding author; stawiski@agh.edu.pl}
}

\affil{AGH University,\\ Faculty of Applied Mathematics, \protect\\al. Mickiewicza 30, 30-059 Krakow, Poland}

\begin{document}
\maketitle
\begin{abstract}
For a graph $G$, we define a small automorphism as one that maps some vertex into its neighbour. We investigate the edge colourings of $G$ that break every small automorphism of $G$. We show that such a colouring can be chosen from any set of lists of length three. In addition, we show that any set of lists of length two on both edges and vertices of $G$ yields a total colouring which breaks all the small automorphisms of $G$. These results are sharp and they match the non-list variants. 


\end{abstract}

\section{Introduction}

The concept of \emph{distinguishing} vertex colourings was introduced in 1977 by Babai \cite{BAB}, as the colourings which are preserved only by the identity automorphism of the graph. The natural optimization problem is to minimize the number of colours in such a colouring, and this minimum number for a given graph $G$ is called the \emph{distinguishing number} of $G$, and denoted by $D(G)$.

We study the following problem proposed by Kalinowski, Pil\'sniak and Wo{\'z}niak \cite{small}. An automorphism $\varphi$ of a graph $G$ is \emph{small} if, for some vertex $v\in V(G)$, $\varphi(v)$ is a neighbour of $v$. We are interested in the minimum number of colours needed to break every small automorphism of $G$, which is called the \emph{small distinguishing index} of $G$ and denoted by $D'_s(G)$. This problem is strongly connected with the concept of general distinguishing of adjacent vertices by edge colourings, where the distinguishing condition is generally stronger than just breaking the automorphisms. In particular, we may demand that the incident vertices have different sums (or sets, or multisets, etc.) of the colours (which are then restricted to the positive integers) on the incident edges. One of the central problems in this field was 1-2-3 Conjecture posed by Karo\'nski, \L{}uczak and Thomason \cite{KLT}, which states that the set of colours $\{1,2,3\}$ is sufficient for any finite graph without $K_2$ as a component to admit an edge colouring so that the sums of colours on the incident edges are different for any two neighbouring vertices. This conjecture was recently confirmed by Keusch \cite{Keusch}, and it was then generalized to locally finite graphs by Stawiski \cite{StawiskixD}. 

Kalinowski, Pil\'sniak and Wo{\'z}niak \cite{small} proved that $D'_s(G) \leq 3$ for any finite graph $G$ without $K_2$ as a component. They also showed that only two colours are sufficient for total colourings. We generalize both these results into the list colouring setting. Moreover, our result holds for both finite and infinite graphs.



We follow the notation in \cite{small} and extend it to the list colourings. In particular, we denote by $D'_{l,s}(G)$ the \emph{small list distinguishing index} of $G$, i.e. the least size of the lists assigned to the edges of $G$ such that there exists an edge colouring from any set of lists of this size which breaks every small automorphism of $G$. We prove that $D'_{l,s}(G)$ is at most three for any finite or infinite locally finite graph without components isomorphic to $K_2$, and we prove that two colours suffices for the analogous problem for total colourings. This is a support for the list version of 1-2-3 Conjecture, and the total list version of 1-2-3 Conjecture, which still remain open.
Note that both results are sharp because $D'_{l,s}(G)=3$ for $G\in \{K_3,K_4,K_5,C_4,C_5\}$. 

\section{Edge colourings}

Let $G$ be an arbitrary graph and $r$ be a vertex of $G$. We say that $(G,r)$ is a \emph{rooted} graph, and we refer to $r$ as the \emph{root} of $G$. The set of automorphisms of $(G,r)$, denoted by $\Aut(G,r)$, is the set of these automorphisms of $G$ which fix $r$. We denote the set of all the automorphisms of $G$ by $\Aut(G)$.

The proof of the main theorems relies on the following lemma, which asserts that the lists of length $2$ are almost sufficient to break all the small automorphisms of any locally finite graph.


\begin{lem} \label{lem:almost}
Let $G$ be a connected locally finite graph other than $K_2$, and let $r$ be an arbitrary vertex of $G$. Then $G$ admits a colouring from the list of length $2$, which breaks all the small automorphisms of $(G,r)$. 
\end{lem}
\begin{proof}

First, consider the case that the graph $G$ has bounded degree. This part of the proof is by induction on $\Delta(G)$. If $\Delta(G)\leq 2$, then $G$ is the  ray, the double ray, a path, or a cycle, and it is easy to verify that the claim holds. Assume then that $\Delta(G)\geq 3$. 

Choose a vertex $r\in V(G)$. We shall construct a colouring of the edges of $G$ that breaks all small automorphisms in $\Aut(G,r)$. Let $\mathcal{A}$ be the set of orbits with respect to $\Aut(G,r)$. We fix some ordering $\mathcal{A} = \{A_0, A_1, \dots\}$ such that if $i<j$, then the vertices of $A_i$ have less or equal distance from $r$ than the vertices of $A_j$ (note that this distance is constant within each orbit). In particular, $A_0=\{r\}$. We shall process these orbits one by one, starting with $A_1$, and in each step $i$ we shall choose colours for all the edges between the vertices of $A_i$ and the vertices of $\bigcup_{j\leq i} A_j$. 

Let $A_i$ be the currently processed orbit. The vertices of $A_i$ have at least one back edge, i.e. an edge to already coloured orbits. Therefore, $\Delta(G[A_i])< \Delta(G)$ and we can colour the edges of each component $H$ of $G[A_i]$ using the induction hypothesis, and break all the small automorphisms of $(H, r_H)$ for some $r_H\in V(H)$. We do not require these components to have non-isomorphic colourings, because we just need to break the small automorphisms in each component. Then, in each component $H$, we must fix the vertex $r_H$ to obtain a colouring that breaks all the small automorphisms of $H$. We achieve it by colouring some edge incident to $r_H$, which lies on a shortest path from this vertex to $r$ (the other end-vertex of that edge is in a previous orbit) using one of the two colours from its list, say blue, and then colouring all the other edges from the same component to the previous orbits with arbitrary colours other than blue. 

We now argue that after repeating these steps for each orbit $A_i$, $i\ge 1$, we break all the small automorphisms of $(G,r)$. Let $\varphi$ be a small automorphism of $G$ that fixes $r$. Then, there is a vertex $x$ such that $\varphi(x)\in N(x)$. These two vertices $x$ and $\varphi(x)$ are contained in some orbit $A_i$ with respect to $\Aut(G,r)$, and since they are neighbours, they must lie in the same component $H$ of this orbit. As our colouring restricted to $H$ breaks all the small automorphisms of $(H,r_H)$, and $r_H$ is the only vertex in $H$ with a blue edge to a previous orbit, then $\varphi$ must change a colour of either this blue edge or one of the edges inside $H$. This means that $\varphi$ cannot be preserved by the colouring.

Finally, if $G$ has vertices of arbitrarily large degrees, then we perform the same procedure as above, and at the point when we used the induction hypothesis, we just use the claim for a finite maximum degree as all the considered orbits are finite.
\end{proof}

\begin{thm}
Let $G$ be a 
locally finite graph without a $K_2$ component. Then $D'_{l,s}(G) \leq 3$. 

\end{thm}
\begin{proof}

Let $G=(V,E)$ be a locally finite graph with no $K_2$ components, and let $\Delta(G)$ be its maximum degree. Note that it is sufficient to break only the automorphisms of each connected component of $G$, since any automorphism that maps one component into another is a composition of a small automorphism that stabilizes all the components or the identity, and a non-small automorphism. Therefore, we shall colour each component of $G$ separately. 
Let $H$ be a component of $G$. 

If $\Delta(H)\leq 2$, then again the proof is straightforward. We shall assume that $\Delta(H)\geq 3$.

Choose a vertex $r\in V(H)$ of degree at least 3. Take any edge incident to $r$ and name any colour from its list pink (without colouring the edge, at this point we just choose a name for the colour). Remove the colour pink from all the lists and use Lemma \ref{lem:almost} to obtain a colouring $c$ of the subgraph $H$ from the modified lists which breaks all the small automorphisms of $(H,r)$. We shall now make a slight correction of this colouring to fix $r$ and keep the distinction of the rest of the subgraph.

If $r$ is fixed by now, then $c$ is already the desired colouring. Otherwise, we shall recolour one of the edges incident to $r$ with pink. We consider all such recolourings, for all the neighbours of $r$. If there is a neighbour $x$ of $r$ such that recolouring the edge $rx$ with pink yields a colouring that fixes $r$ with respect to $\Aut(H)$, then we implement this change and return the resulting colouring. Assume now that no such neighbour exists. If all the edges incident to $r$ are of the same colour, say blue, then we recolour an arbitrary edge $rx$ with pink and some other edge $ry$ with a colour other than pink and blue, say red. Note that $r$ and $x$ became the only two vertices with a pink incident edge, and $r$, unlike $x$, has an incident red edge, so they are distinguished. If, otherwise, there are at least two different colours on the edges incident to $r$, say $rx$ and $ry$, then we recolour $ry$ with the colour of $rx$ (say red) and $rx$ with pink. Here we must be a bit more careful, and if possible, choose $x$ and $y$ from the same component of the same orbit, so that $x$ is the root of that component. 
This way, we fix $r$ for the exact same reason as in the previous case. 

Let now $\varphi$ be a small automorphism of $H$ which is preserved by the resulting colouring. By the arguments above, $\varphi$ must fix $r$. Since $c$ broke all the small automorphisms of $(H,r)$, there must be an edge $e$ such that $c(e) \neq c(\varphi(e))$. Then, $e$ or $\varphi(e)$ must have changed its colour during the correction, so $e=rx$ and $\varphi(e)=ry$ for some $x,y$. The new colour for $e$ or $\varphi(e)$ cannot be pink, as there is only one pink edge in the resulting colouring. Hence, the new colour must be red, and it is the case that there were at least two different colours on the edges incident to $r$ in the colouring $c$. Then, either $x$ and $y$ are in the same component of the same orbit (and, since $x$ is its root, there would have to be a second edge inside this component, mapped by $\varphi$ into an edge of a different colour), or any component of any orbit contained in $N(r)$ has the same colour on the edges to $r$ (so these colours were not relevant for breaking $\varphi$ and some other edge not incident to $r$ changes colour after applying $\varphi$). In both cases, $\varphi$ cannot be preserved by our colouring.
\end{proof}

\section{Total colourings}

\begin{thm}
Let $G$ be a locally finite graph without a $K_2$ component and $\mathcal{L}=\{L_x\}_{x\in V\cup E}$ be the set of lists of length two for the vertices and edges of $G$. Then, $G$ admits a total colouring from lists, which breaks all the small automorphisms of $G$. 
\end{thm}
\begin{proof}
We again consider each connected component $H$ of $G$ separately. Let $r$ be an arbitrary vertex in $H$. By Lemma \ref{lem:almost}, there is an edge colouring of $H$ from the lists $\{L_e\}_{e\in E}$, which breaks every small automorphism of $(H,r)$. Then, choose the colours for the vertices so that $r$ has a unique colour in the component.  
\end{proof}

\bibliographystyle{abbrv}
\bibliography{lit.bib}

\end{document}